\documentclass[12pt,reqno]{amsart}
\usepackage[latin1]{inputenc}
\usepackage{enumitem}
\usepackage[T1]{fontenc}
\usepackage{amssymb}
\usepackage{placeins}

\usepackage[colorlinks=false]{hyperref}

\usepackage{cleveref}

\usepackage[ruled,vlined,algosection,norelsize]{algorithm2e}
\usepackage{rotating}

\usepackage{tikz}
\usetikzlibrary{calc}

\usepackage{color}

\usepackage[normalem]{ulem}

\newcommand{\bool}[1]{\mathfrak{B}(#1)}

\newtheorem{theorem}{Theorem}[section]
\newtheorem{lemma}[theorem]{Lemma}
\newtheorem{corollary}[theorem]{Corollary}

\newtheorem{conjecture}[theorem]{Conjecture}

\newtheorem*{theorem-non}{Theorem}
\theoremstyle{definition}
\newtheorem{remark}[theorem]{Remark}

\newtheorem{example}[theorem]{Example}

\numberwithin{equation}{section}

\def\ZZ{{\mathbb Z}}
\def\NN{{\mathbb N}}

\def\QQ{{\mathbb Q}}
\def\KK{{\mathbb K}}

\def\sdepth{\operatorname{sdepth}}

\def\depth{\operatorname{depth}}

\DeclareMathOperator{\spdim}{spdim}
\DeclareMathOperator{\pdim}{pdim}
\DeclareMathOperator{\codim}{codim}

\newcommand{\set}[1]{\{#1\}}
\newcommand{\with}{\,:\,}

\newcommand{\defa}{:=}

\newcommand{\qq}[1]{``#1''}

\begin{document}

\title{Lcm-lattices and Stanley depth: a first computational approach}

\author{Bogdan Ichim}

\address{Simion Stoilow Institute of Mathematics of the Romanian Academy, Research Unit 5, P.O. Box 1-764, 014700 Bucharest, Romania} \email{bogdan.ichim@imar.ro}

\author{Lukas Katth\"an}

\address{Universit\"at Osnabr\"uck, FB Mathematik/Informatik, 49069
Osnabr\"uck, Germany}\email{lukas.katthaen@uos.de}

\author{Julio Jos\'e Moyano-Fern\'andez}

\address{Universidad Jaume I, Campus de Riu Sec, Departamento de Matem\'aticas, 12071
Castell\'on de la Plana, Spain} \email{moyano@uji.es}

\thanks{The first author was partially supported  by project  PN-II-RU-TE-2012-3-0161, granted by the Romanian National Authority for Scientific Research,
CNCS - UEFISCDI. The second author was partially supported by
the German Research Council DFG-GRK~1916. The third author was partially supported by the Spanish Government Ministerio de Econom\'ia y Competitividad (MEC), grant MTM2012-36917-C03-03 in cooperation with the European Union in the framework of the founds ``FEDER''; he also thanks the Institute of Mathematics at the University of Osnabr\"uck for hospitality during Summer 2014}

\begin{abstract}
Let $\KK$ be a field, and let $S=\KK[X_1, \ldots , X_n]$ be the polynomial ring. Let $I$ be a monomial ideal of $S$ with up to 5 generators.
In this paper, we present a computational experiment which allows us to prove that
\[
\depth_S S/I = \sdepth_S S/I < \sdepth_S I.
\]
This shows that the Stanley conjecture is true for $S/I$ and $I$, if $I$ can be generated by at most $5$ monomials.
The result also brings additional computational evidence for a conjecture made by Herzog.
\end{abstract}

\maketitle

\section{Introduction}

Let $\KK$ be an infinite field. Let $S=\KK[X_1, \ldots , X_n]$ be the standard $\ZZ^n$-graded polynomial ring, and let $M$ be a finitely generated $\ZZ^n$-graded $S$-module. The \emph{Stanley depth} of $M$, denoted $\sdepth_S M$, is a combinatorial invariant of $M$ related
to a conjecture of Stanley from 1982 \cite[Conjecture 5.1]{St}, which states that the inequality $\depth_S M \leq \sdepth_S M$ holds;  this is nowadays called the \emph{Stanley conjecture}.

In this paper we restrict ourselves to the case of a monomial ideal $I \subset S$.  The least common multiple lattice (or simply the \emph{lcm-lattice}) $L_I$ of $I$---introduced by Gasharov, Peeva and Welker in \cite{GPW}---has been revealed to be a useful tool in the study of the Stanley conjecture: as the authors have shown in \cite{IKM2}, the lcm-lattice $L_I$ essentially determines the Stanley depth (as well as the usual depth) of both $S/I$ and $I$.
More precisely, the lcm-lattice $L_I$ determines the \emph{Stanley projective dimension} in the particular cases $M = I$ or $M=S/I$. This invariant may be defined in general as $\spdim_S M := n - \sdepth_S M$ (see \cite{IKM2} for a precise definition), similar to the usual projective dimension (by the well-known Auslander-Buchsbaum formula \cite[Theorem 1.3.3]{BH}).

Since the Stanley projective dimension of an ideal depends only
on its lcm-lattice, one can interpret this number as a combinatorial invariant of
the lattice itself. This leads to a purely lattice theoretical formulation of the Stanley
conjecture. More precisely, let $L$ be an atomistic lattice. According to \Cref{prop:realize}, there exists an ideal $I \subset S = \KK[X_1,\dotsc,X_n]$ such that $L_I \cong L$ and we set
$\pdim_1 L \defa \pdim I$, $\pdim_2 L \defa \pdim S/I$, $\spdim_1 L \defa \spdim I$, and $\spdim_2 L \defa \spdim S/I$.
Then the statements \cite[Conjecture~2]{A1}, \cite[Conjecture~1]{A2} and \cite[Conjecture 64]{H} can be reformulated as follows:

\begin{conjecture}\cite[Conjecture 6.2]{IKM2} \label{conj:stanleylattice}
	For all finite lattices $L$, the following inequalities hold:
	\begin{enumerate}
		\item $\spdim_1 L \leq \pdim_1 L$,
		\item $\spdim_2 L \leq \pdim_2 L$, and
		\item $\spdim_1 L \leq \spdim_2 L - 1$.
	\end{enumerate}
\end{conjecture}

The use of lcm-lattice techniques makes possible a rather simple computational approach to Conjecture \ref{conj:stanleylattice} for ideals with few generators, which is presented in this paper.

As a result of our computational experiments, we prove the following:

\begin{theorem} \label{thm:main}
For every monomial ideal $I \subsetneq S$ with at most $5$ generators, it holds that
\[
\depth_S S/I = \sdepth_S S/I < \sdepth_S I.
\]
This further implies that all parts of Conjecture \ref{conj:stanleylattice} hold (in this case).
\end{theorem}
\bigskip

The content of the paper is organized as follows. Section \ref{sec:pre} includes some needed prerequisites (concerning Stanley depth and lattice theory).
In Section \ref{sec:enum} we present an algorithm for computing all lcm-lattices of monomial ideals with a given number of minimal generators.
This allows the computation of the Stanley depth for all lcm-lattices, up to a certain bound.
In Section \ref{sec:sdepth} we classify all lcm-lattices of monomial ideals with at most 5 generators with the aid of the computer and we compute several attached invariants.
This leads to Theorem \ref{thm:main}; observe that in this case the \emph{equality} $\depth_S S/I = \sdepth_S S/I$ holds.
We also provide an example of an ideal with 6 generators where the equality no longer holds (see Example \ref{ex:6generators}).
We finish the paper by presenting the full numerical data obtained during our experiments for all lcm-lattices of monomial ideals with 4 generators in the Appendix. We hope this data can be useful also for other people.

For the lattice theoretical computations we used the Maple package \texttt{poset} by J. Stembridge \cite{JS}.
For the computation of invariants, we used \texttt{Macaulay2} \cite{M2} and our own experimental software implemented in C++ and Maple.

\section{Preliminaries} \label{sec:pre}

In this section we present the basics of the two main ingredients needed in the paper, namely lcm-lattices and Stanley depth.

\subsection{Stanley depth and Stanley projective dimension}

Consider the polynomial ring $S$ endowed with the multigraded structure. Let $M$ be a finitely generated graded $S$-module, and let $\lambda$ be a homogeneous element in $M$.
Let $Z \subset \set{X_1, \ldots , X_n}$ be a subset of the set of indeterminates of $S$. The $\KK[Z]$-submodule $\lambda \KK[Z]$ of $M$ is called a \emph{Stanley space} of $M$ if $\lambda \KK[Z]$ is free (as $\KK[Z]$-submodule).
A \emph{Stanley decomposition} of $M$ is a finite family
$$
\mathcal{D}=(\KK[Z_i],\lambda_i)_{i\in \mathcal{I}}
$$
in which $Z_i \subset \{X_1, \ldots ,X_n\}$ and $\lambda_i\KK[Z_i]$ is a Stanley space of $M$ for each $i \in \mathcal{I}$ with
$$
M = \bigoplus_{i \in \mathcal{I}} \lambda_i\KK[Z_i]
$$
as a multigraded $\KK$-vector space. This direct sum carries the structure of $S$-module and has therefore a well-defined depth (the $S$-module structure of the direct
sum is different from that of $M$ in general). The \emph{Stanley depth} $\sdepth\ M$ of $M$ is defined to be the maximal depth of a Stanley decomposition of $M$.

\subsection{The lcm-lattice of a monomial ideal}
The second ingredient we need is a bit of lattice theory.

A \emph{lattice} $L$ is a partially ordered set $(L,\leq)$ such that, for any $P,Q \in L$, there is a unique greatest lower bound $P \wedge Q$ called the meet of $P$ and $Q$, and there is a unique least upper bound $P\vee Q$ called the join of $P$ and $Q$.

In this paper we only consider \emph{finite} lattices, so in the remainder of the paper all lattices are assumed to be finite.
A finite lattice has, in particular, a unique minimal element $\hat{0}$ and a unique maximal element $\hat{1}$.

A \emph{meet-irreducible} element of $L$ is an element which is covered by exactly one other element.
An \emph{atom} is an element in $L$ that covers the minimal element $\hat{0}$. 
We say that a lattice is \emph{atomistic}, if every element can be written as a join of atoms.
A \emph{join-preserving map} $\phi: L \longrightarrow L'$ is a map with $\phi(P \vee Q)=\phi(P) \vee \phi(Q) $ for all $P,Q \in L$.

Let $I \subset S$ be a monomial ideal with $I \neq (0)$. The \emph{lcm-lattice of $I$}, $L_I$, is defined as the set of all monomials that can be obtained as the least common multiple (lcm) of some non-empty subset of the minimal generators of $I$, ordered by divisibility.
Note that $L_I$ is an atomistic lattice and that its atoms are exactly the minimal generators of $I$.
Conversely, the following holds (see \cite{P}, and \cite[Corollary 3.6]{IKM2}):

\begin{corollary}\label{prop:realize}
Every finite atomistic lattice can be realized as lcm-lattice of monomial ideal.
\end{corollary}

The starting point for the considerations in the present paper is the following result:
\begin{theorem}[\cite{GPW},\cite{IKM2}]\label{thm:lcmmap}
	Let $I \subset S$ and $I' \subset S'$ be two monomial ideals in two (possibly) different polynomial rings.
	Assume that there is a surjective join-preserving map $L_I \rightarrow L_{I'}$, which is bijective on the atoms.
	Then
\begin{enumerate}
\item $\spdim_S S/I \geq \spdim_{S'} S'/I'$;
\item $\spdim_S I \geq \spdim_{S'} I'$;
\item $\pdim_S S/I \geq \pdim_{S'} S'/I'$ and
\item $\pdim_S I \geq \pdim_{S'} I'$.
\end{enumerate}
In particular, $\spdim_S S/I$, $\spdim_S I$,
$\pdim_S S/I$ and $\pdim_S I$ depend only on $L_I$.
\end{theorem}
Here, the statement about the projective dimension follows from Theorem 3.3 of \cite{GPW}.
The statements about the Stanley projective dimension are special cases of Theorem 4.5 in \cite{IKM2} by the authors.

\begin{remark}
Note that in \cite{IKM2}, we considered modules of the form $I/J$, which required us to work in the more general context of
pairs of lcm-\emph{semi}lattices of sets of monomials.
In the present paper, we only consider modules of the form $S/I$ and $I$, so we can simplify the discussion.
\end{remark}

The next two structural lemmata are Lemma 4.3 and Lemma 4.4 in \cite{IKM2}, and they will be needed in the sequel.
Fix a meet-irreducible element $a \in L$ and let $a_+ \in L$ denote the unique element covering it.
Consider the equivalence relation $\sim_a$ on $L$ defined by setting $a \sim_a a_+$ and any other element is equivalent only to itself.
\begin{lemma}\label{lem:homom}
	There is a natural lattice structure on $L / \sim_a$, such that the canonical surjection $\pi_a: L \longrightarrow L/\sim_a$ preserves the join.
	Moreover, if $L$ is atomistic and $a$ is not an atom, then $L / \sim_a$ is atomistic.
\end{lemma}

\begin{lemma}\label{lem:factor}
Let $L, L'$ be finite lattices and $\phi: L \longrightarrow L'$ a join-preserving map.
\begin{enumerate}
	\item If $\phi$ is not injective, then there exists a meet-irreducible element $a \in L$ such that $\phi(a) = \phi(a_+)$.
	\item If $\phi(a) = \phi(a_+)$ for some meet-irreducible element $a \in L$, then $\phi$ factors through $L / \sim_a$.
\end{enumerate}
\end{lemma}

\section{An algorithm for computing lcm-lattices} \label{sec:enum}

The aim of this section is to present an algorithm for computing all lcm-lattices of monomial ideals with a given number of minimal generators.

For $k \in \NN \setminus\set{0}$ let $\bool{k}$ denote the lattice of subsets of a $k$-element set, i.e. the boolean lattice on $k$ atoms.
Note that $\bool{k}$ is the lcm-lattice of an ideal generated by $k$ indeterminates.
First of all, we make the following simple observation.
\begin{remark}\label{rem:free}
	For every atomistic lattice $L$ on $k$ atoms, there exists a surjective join-preserving map $\phi: \bool{k} \longrightarrow L$. The map $\phi$ may be constructed as follows. Let $\phi$ map the atoms of $\bool{k}$ bijectively on the atoms of $L$. For every other element $a \in \bool{k}$ we set $\phi(a) := \phi(a_1) \vee \dotsb \vee \phi(a_l)$ where $a = a_1 \vee a_2 \vee \dotsb \vee a_l$ is the unique (up to order) way to write $a$ as a join of atoms.
\end{remark}

Let $\mathcal{L}_k$ denote the set of all isomorphism classes of finite atomistic lattices with (exactly) $k$ atoms. Consider the following partial order on $\mathcal{L}_k$: We set $L \geq L'$ if and only if there exists a surjective join-preserving map $L \rightarrow L'$.
By \Cref{lem:homom} and \Cref{lem:factor}, this poset is graded by the cardinality of the lattices and the cover relations correspond to maps of the form $\pi_a: L \rightarrow L / \sim_a$.
We use these facts to obtain an algorithm (see Algorithm \ref{Alg:enum}) for generating all atomistic lattices with a given number of atoms. By \Cref{thm:lcmmap}, this algorithm implicitly offers a full classification in respect with  $\spdim_S S/I$, $\spdim_S I$, $\pdim_S S/I$ and $\pdim_S I$.
\medskip

\allowdisplaybreaks
\begin{algorithm}[H]
\SetKwFunction{Union}{{\bf Union}}
\SetKwFunction{Beg}{begin}
\SetKwFunction{En}{end}
\SetKwFunction{size}{size}
\SetKwFunction{Insert}{Insert}
\SetKwFunction{GenerateLattices}{{\bf GenerateLattices}}
\SetKwFunction{GenerateL}{{\bf GenerateL}}
\SetKwFunction{FactorizeModulo}{{\bf FactorizeModulo}}
\SetKwFunction{DeleteDuplicates}{{\bf DeleteDuplicates}}
\SetKwData{Container}{Container}
\SetKwData{Lattice}{Lattice}
\caption{Generating atomistic lattices on $k$ atoms}\label{Alg:enum}

\KwData{$k \in \ZZ, \ k\ge 1$}
\KwResult{A container containing the atomistic lattices on $k$ atoms, i.e. $\mathcal{L}_k$}
\Container \GenerateLattices{$k$}\;
\Begin{
\Container $R,W$\;
\Lattice $L,F$\;
\nl $W=\Insert{W,\GenerateL{k}}$\;
\nl \While {$\size{W}>0$}{
\nl    $R=\Union{R,W}$\;
    \Container $V$\;
\nl \For {$ i=\Beg{W}\ \KwTo\ i=\En{W}$ }{
    $L=W[i]$\;
\nl \For {$a\in L\ {\bf and }\ a\ \text{meet-irreducible}\ {\bf and }\ a\ \text{not an atom}$}{
\nl $P=\FactorizeModulo{L,a}$\;
    $V=\Insert{V,P}$\;
}
}
\nl $W=\DeleteDuplicates(V)$\;
}
\nl \Return $R$\;
}
\end{algorithm}

The container data type used for the implementation of the algorithm should allow basic access functions like \texttt{begin, end, size} (for example we want to ask for its size), insertion of an element or the union of the two containers.
We also assume that a data structure \texttt{Lattice} (suited for storing lattices) is given.

Below we describe the key steps in the algorithm.
\begin{itemize}
	\item line 1. The function {\bf GenerateL}  generates the lattice $\bool{k}$, which is further used to initialize the container $W$;
	\item line 2. Main loop, at each iteration the cardinality of the lattices contained in the container $W$ decreases by one.
		So after finitely many iterations, $W$ will be empty and the algorithm terminates.
    \item line 3. The results of a previous iteration of the main loop are stored in the container $R$.
	\item line 4. In this loop we analyze all lattices contained in $W$.
	\item line 5. We search for all possible factorizations of a given lattice $L$.
	\item line 6. If a factorization is possible, it is computed using the function {\bf FactorizeModulo} and stored in the container $V$.
	\item line 7. For each isomorphism class in $V$, we keep only one representative and append this to $W$. Then the main loop can be repeated.
    \item line 8. The container $R$ storing the generated lattices is returned.
\end{itemize}

\begin{theorem}
	Algorithm \ref{Alg:enum} returns $\mathcal{L}_k$.
\end{theorem}
\begin{proof}
	Every lattice that is produced by Algorithm \ref{Alg:enum} is atomistic and has $k$ atoms by \Cref{lem:homom}.
	
	On the other hand, let $L$ be an atomistic lattice on $k$ atoms.
	By Remark \ref{rem:free}, there exists a surjective join-preserving map $\phi: \bool{k} \rightarrow L$.
	By repeatedly applying \Cref{lem:factor}, we obtain a factorization of $\phi$ as
	\[\bool{k} \rightarrow L' \stackrel{\psi}{\rightarrow} L \]
	where $L'$ is a lattice produced by the algorithm and $\psi$ is bijective.
	
	Finally, the container returned by Algorithm \ref{Alg:enum} contains no duplicates, because in each iteration of the \texttt{while}-loop, the cardinality of the lattices contained in $W$ decreases by one.
\end{proof}

\section{Computational experiments -- lcm-lattices with few generators} \label{sec:sdepth}

In our paper \cite{IKM2} we gave a purely lattice theoretical formulation of the Stanley conjecture, as we saw in the Introduction. So turning to Conjecture \ref{conj:stanleylattice}, let us first relate this to the Stanley conjecture for monomial ideals. As a consequence of \Cref{thm:lcmmap} we obtain:

\begin{corollary}\label{coro:lcmmap}
The functions $\spdim_1, \spdim_2, \pdim_1, \pdim_2: \mathcal{L}_k \rightarrow \NN$ are all monotonous.
\end{corollary}

This allows the following strategy to computationally prove the Stanley conjecture for monomial ideals with a fixed number $k$ of generators:
First, we generate $\mathcal{L}_k$ with a modified version of Algorithm \ref{Alg:enum}, where we keep track of the relations between the lattices.
We implemented Algorithm \ref{Alg:enum} in Maple using the \texttt{poset}-package by J. Stembridge \cite{JS}.
With our implementation we were able to enumerate $\mathcal{L}_k$ for $k \leq 5$. The number of isomorphism classes is given in Table \ref{table:number}.

\begin{table}[h]
\begin{tabular}{|l|l|l|}
\hline $k$ & $\#$Lattices & Computation times \\
\hline $2$ & $1$ & - \\
\hline $3$ & $4$ & $0.05$ sec\\
\hline $4$ & $50$ & $2$ sec\\
\hline $5$ & $7443$ & $35$ min\\
\hline $6$ & $>75.000$ & $> 1$ week\\
\hline
\end{tabular} \vspace{0.3cm}
\caption{The numbers of atomistic lattices with $k$ generators.}\label{table:number}
\end{table}

For $k \leq 5$, the experiments were run on a
Lenovo Thinkpad computer with a Intel Core i5-2450M processor running at 2.50 GHz with 4 GB of RAM.
For $k=6$, we used a desktop computer with similar specifications and we stopped the algorithm after one week.
\medskip

Then we compute $\pdim_1$ for all elements of $\mathcal{L}_k$.
For this, we choose an ideal $I$ with $L_I \cong L$ for each $L \in \mathcal{L}_k$ and compute its projective dimension using the computer algebra package \texttt{Macaulay2} \cite{M2}.
For $k = 5$, this took only a few seconds.

For each fixed value of $\pdim_1$ (resp.~$\pdim_2$), we select from $\mathcal{L}_k$ the maximal elements with this value.
These are much fewer lattices, e.g. there are 8 cases for $k=5$, which should be compared with the total number $|{\mathcal{L}_k}| = 7443$.
Then, for each of these \qq{extremal} lattices, we verify the Stanley conjecture using a fast C++ implementation of the algorithms described in \cite{IJ} and \cite{IZ}.
By \Cref{coro:lcmmap}, that is enough to prove the Stanley conjecture for all elements of $\mathcal{L}_k$; this is the content of the already mentioned Theorem \ref{thm:main}:

\begin{theorem-non}[\ref{thm:main}]
	For every monomial ideal $I \subsetneq S$ with at most $5$ generators, it holds that
	\[ \depth_S S/I = \sdepth_S S/I < \sdepth_S I. \]
	This further implies that all parts of Conjecture \ref{conj:stanleylattice} hold in this case.
\end{theorem-non}
\begin{proof}
	We computationally verified the inequalities $\depth_S S/I \leq \sdepth_S S/I$ and $\depth_S I \leq \sdepth_S I$ by the method described above.
	We further verified the inequality $\depth_S S/I \geq \sdepth_S S/I$, using that it is enough to verify it at \emph{minimal} elements of $\mathcal{L}_k$ with a given projective dimension.
\end{proof}

Finally, we present an example to show that the equality $\sdepth_S S/I = \depth_S S/I$ does not hold in general for ideals with more than five generators.
\begin{example}\label{ex:6generators}
	Consider the ideal $I := (x_1x_3,x_2x_3,x_1x_4,x_2x_4,x_1x_5,x_2x_5)$ in the ring $\QQ[x_1,\dotsc,x_5]$.
	It is the Stanley-Reisner ideal of the disjoint union of a triangle and an edge. As this simplicial complex is not connected, the depth of $S/I$ is $1$ (for a brief explanation of this fact, see \cite[Exercise 5.1.26]{BH}).
	On the other hand, the following Stanley decomposition of $S/I$ shows that its Stanley depth is $2$:
	\[ \QQ[x_1,x_2] \oplus x_3 \QQ[x_3,x_4] \oplus x_4 \QQ[x_4,x_5]  \oplus x_5 \QQ[x_5,x_3] \oplus x_3x_4x_5\QQ[x_3,x_4,x_5]. \]
\end{example}

\begin{remark}
From Theorem \ref{thm:main}, we deduce that for every monomial ideal $I \subsetneq S$ with at most $5$ generators, it holds that
\[
\sdepth_S I \geq \sdepth_S S/I + 1,
\]
confirming Herzog's conjecture \cite[Conjecture~1.64]{H} (and the remark after it).
\end{remark}

\section*{Acknowledgements}
The authors would like to thank the anonymous reviewer for several helpful suggestions.

\section*{Appendix: Invariants of some lattices}

In this appendix we include some of the numerical data obtained during our computational experiments. We think they can be as useful for the reader as they were for us.
Table \ref{table:ideals} shows for each $L \in \mathcal{L}_4$ an example of an ideal $I$ such that $L_I \cong L$. 
These ideals were found using Theorem 3.4 of \cite{IKM2}.

Tables \ref{table:first25} and \ref{table:last25} record some invariants associated to lattices $L_I$, namely the projective and Stanley projective dimensions of both $S/I$ and $I$ and the order-dimension $\dim L$ (which is defined to be the smallest cardinal $m$ such that $L$ is a sublattice of a product of $m$ chains).

For each $L$, we also give the maximum of the codimensions of $S/I$ of all ideals $I$ with $L \cong L_I$. We denote this by
\[ \max \codim L := \max\set{\codim S/I \with L_I \cong L}. \]
Moreover, the length and breadth of a lattice $L$ are shown. Just for the sake of completeness: Let $n,p$ be natural numbers. A lattice $L$ is said to be of length $n$ if there is a chain in $L$ of length $n$ and all chains in $L$ are of length $\leq n$. The lattice is said to be of breadth $p$, if $p$ is the smallest integer with the property that for every nonempty subset $X \subseteq L$, there exists a nonempty subset $Y \subseteq X$ such that $|Y|\leq p$ and $\bigvee X =\bigvee Y$. For further details the reader is referred to the book of Gr\"atzer \cite{Gr}.

Finally, we indicate for each lattice $L$ whether there \emph{exists} a monomial ideal $I$ with $L \cong L_I$ which is Cohen-Macaulay or strongly generic.
Here, an $X$ indicates that there exists an ideal with the given property, while a blank space indicates that such an ideal does not exist. For details on the Cohen-Macaulay property, we refer the reader to the monograph \cite{BH} dedicated to this subject.
Note that even in these cases, in general not all ideals $I$ with $L \cong L_I$ satisfy the respective property.
The reason for this is that the projective dimension is determined by the lcm-lattice, whereas the codimension is not.
Moreover, for a given lattice $L$, there exists a Cohen-Macaulay monomial ideal $I$ with $L_I \cong L$ if and only if $\max \codim L = \pdim S/I$.
We use the definition of strongly generic as given in \cite{BPS} (just called \emph{generic} there), i.e.\ no variable should appear with the same exponent in two different minimal generators.
We distinguish it from the general notion introduced in \cite{MSY}.
A handy criterion for the existence of a strongly generic monomial ideal with a given lcm-lattice is given in \cite[Theorem 6.2]{M}.

For the convenience of the reader we summarize the known relations among the invariants. It is well-known that $\pdim S/I = \pdim I + 1$. The other relations are

\[2 \leq \max \codim L \leq \set{\spdim S/I, \pdim S/I} \leq \set{\mathrm{length} L, \dim L} \leq k = 4 \]
\[1 \leq \lfloor\frac{\max \codim L}{2}\rfloor \leq \spdim I \leq \set{\mathrm{length} L  - 1, \dim L - 1, \lfloor\frac{k}{2}\rfloor = 2}. \]

Here, we write $a \leq \set{b,c} \leq d$ to express the two inequalities $a \leq b \leq d$ and $a\leq c \leq d$, while no relation between $b$ and $c$ is known.
These inequalities hold in general for all ideals $I$ with $L \cong L_I$ and can be found in \cite[Proposition 1.2.13]{BH}, \cite[Theorem 9]{H}, \cite{KSF}, \cite{O} and \cite[Theorem 5.2, Corollary 5.9]{IKM2}.
%

\allowdisplaybreaks

\begin{table}[h]
\begin{minipage}{.5\linewidth}
\begin{tabular}{c | c}\hline
No. & Ideal  \\ \hline
1 & $(x_4, x_3, x_2, x_1)$ \\ \hline
2 & $(x_3^2, x_2^2, x_1^2, x_1x_2x_3)$ \\ \hline
3 & $(x_2^2, x_1^2, x_3, x_1x_2)$ \\ \hline
4 & $(x_3^2x_4, x_2^2x_5, x_1x_3x_5, x_1x_2x_4)$ \\ \hline
5 & $(x_2^2, x_1x_3, x_1^2, x_1x_2)$ \\ \hline
6 & $(x_2^2x_3, x_1^2x_4, x_2x_4, x_1x_3)$ \\ \hline
7 & $(x_3^2x_4, x_1x_2x_3, x_2^2, x_1x_4)$ \\ \hline
8 & $(x_4^2x_5x_6, x_2x_3x_4, x_1x_3x_5, x_1x_2x_6)$ \\ \hline
9 & $(x_2^2x_3, x_1x_3^2, x_1^2x_2, x_1x_2x_3)$ \\ \hline
10 & $(x_2^2x_3, x_1^3, x_1^2x_2, x_1x_3)$ \\ \hline
11 & $(x_2x_3x_4, x_1x_4^2, x_1^2x_3, x_1x_2)$ \\ \hline
12 & $(x_3^2x_4x_5, x_1x_2x_3, x_2x_4, x_1x_5)$ \\ \hline
13 & $(x_2^2x_3, x_1x_3, x_1x_2, x_4)$ \\ \hline
14 & $(x_3^2, x_2^2, x_1x_3, x_1x_2)$ \\ \hline
15 & $(x_1x_3x_4, x_1x_2x_5, x_4^2x_5, x_2x_3)$ \\ \hline
16 & $(x_2x_4x_5, x_2x_3x_6, x_1x_4x_6, x_1x_3x_5)$ \\ \hline
17 & $(x_2^3x_3, x_1^3x_3, x_1^2x_2^2, x_1x_2x_3)$ \\ \hline
18 & $(x_2^2x_3, x_1^2, x_1x_3, x_1x_2)$ \\ \hline
19 & $(x_2^3, x_1^3, x_1^2x_2, x_1x_2^2)$ \\ \hline
20 & $(x_2^2x_3x_4, x_1^2x_3, x_1^2x_2, x_1x_4)$ \\ \hline
21 & $(x_2x_3x_4, x_1^3x_4, x_1^2x_2, x_1x_3)$ \\ \hline
22 & $(x_2x_3, x_1x_3^2, x_1^2, x_1x_2)$ \\ \hline
23 & $(x_3^2x_4, x_2x_3, x_1x_4, x_1x_2)$ \\ \hline
24 & $(x_1x_3x_4, x_1x_2x_5, x_3x_5, x_2x_4)$ \\ \hline
25 & $(x_1x_2x_4, x_2x_3, x_1x_3, x_4^2)$ \\ \hline
\end{tabular}
%
\end{minipage}\begin{minipage}{.5\linewidth}
\begin{tabular}{c | c}\hline
No. & Ideal \\ \hline
26 & $(x_2^3x_3, x_1^2x_3, x_1x_2^2, x_1x_2x_3)$ \\ \hline
27 & $(x_2^3x_3^2, x_1^2x_3^2, x_1^2x_2^2, x_1x_2x_3)$ \\ \hline
28 & $(x_2^2x_3x_4, x_1x_4, x_1x_3, x_1x_2)$ \\ \hline
29 & $(x_2x_3x_4, x_1^2x_4, x_1x_3, x_1x_2)$ \\ \hline
30 & $(x_2^3x_3, x_1^2x_3, x_1x_2^2, x_1^2x_2)$ \\ \hline
31 & $(x_2x_3x_4, x_1^2x_3, x_1^2x_2, x_1x_4)$ \\ \hline
32 & $(x_2x_3, x_1^3, x_1^2x_2, x_1x_3)$ \\ \hline
33 & $(x_2x_3, x_1x_3, x_1x_2^2, x_1^2x_2)$ \\ \hline
34 & $(x_2x_4, x_2x_3, x_1x_4, x_1x_3)$ \\ \hline
35 & $(x_2x_3, x_1x_3, x_1x_2, x_4)$ \\ \hline
36 & $(x_2^2x_3x_4, x_1^2x_3x_4, x_1x_2x_4, x_1x_2x_3)$ \\ \hline
37 & $(x_2^2x_3^2, x_1^2x_3^2, x_1^2x_2, x_1x_2x_3)$ \\ \hline
38 & $(x_2^3x_3^2, x_1x_3^2, x_1x_2^2, x_1x_2x_3)$ \\ \hline
39 & $(x_2^2x_3^2, x_1^2x_3^2, x_1^2x_2^2, x_1x_2x_3)$ \\ \hline
40 & $(x_2x_3x_4, x_1x_4, x_1x_3, x_1x_2)$ \\ \hline
41 & $(x_2x_3, x_1^2, x_1x_3, x_1x_2)$ \\ \hline
42 & $(x_2^3x_3, x_1^2x_3, x_1^2x_2, x_1x_2^2x_3)$ \\ \hline
43 & $(x_2^2x_3, x_1^2x_3, x_1^2x_2, x_1x_2^2)$ \\ \hline
44 & $(x_2^2x_3^2x_4, x_1x_3^2x_4, x_1x_2x_4, x_1x_2x_3)$ \\ \hline
45 & $(x_2^2x_3, x_1^2x_3, x_1^2x_2, x_1x_2x_3)$ \\ \hline
46 & $(x_2^2x_3^2, x_1x_3^2, x_1x_2^2, x_1x_2x_3)$ \\ \hline
47 & $(x_2^2x_3x_4, x_1^2x_3x_4, x_1^2x_2x_4, x_1x_2^2x_3)$ \\ \hline
48 & $(x_2^2x_3x_4, x_1x_3x_4, x_1x_2x_4, x_1x_2x_3)$ \\ \hline
49 & $(x_2x_3^2x_4, x_1x_3^2x_4, x_1x_2x_4, x_1x_2x_3)$ \\ \hline
50 & $(x_2x_3x_4, x_1x_3x_4, x_1x_2x_4, x_1x_2x_3)$ \\ \hline
\end{tabular}
\end{minipage}
\vspace*{2ex} \caption{The ideals} \label{table:ideals}
\end{table}
\
\medskip
\
\begin{sidewaystable}[h]
\vspace{16cm}

\begin{tabular}{*{26}{c}}\hline
No. 			& 1 & 2 & 3 & 4 & 5 & 6 & 7 & 8 & 9 & 10 & 11 & 12 & 13 & 14 & 15 & 16 & 17 & 18 & 19 & 20 & 21 & 22 & 23 & 24 & 25\\ \hline
$|L|$ 			& 16 & 15 & 14 & 14 & 13 & 13 & 13 & 13 & 12 & 12 & 12 & 12 & 12 & 12 & 12 & 12 & 11 & 11 & 11 & 11 & 11 & 11 & 11 & 11 & 11\\ \hline
pdim $S/I$		& 4 & 3 & 3 & 3 & 3 & 3 & 3 & 3 & 2 & 3 & 3 & 3 & 3 & 3 & 3 & 3 & 2 & 3 & 2 & 3 & 3 & 3 & 3 & 3 & 3\\
spdim $S/I$		& 4 & 3 & 3 & 3 & 3 & 3 & 3 & 3 & 2 & 3 & 3 & 3 & 3 & 3 & 3 & 3 & 2 & 3 & 2 & 3 & 3 & 3 & 3 & 3 & 3\\ \hline
pdim $I$		& 3 & 2 & 2 & 2 & 2 & 2 & 2 & 2 & 1 & 2 & 2 & 2 & 2 & 2 & 2 & 2 & 1 & 2 & 1 & 2 & 2 & 2 & 2 & 2 & 2\\
spdim $I$		& 2 & 2 & 2 & 2 & 1 & 2 & 2 & 2 & 1 & 1 & 1 & 2 & 2 & 1 & 2 & 2 & 1 & 1 & 1 & 1 & 1 & 1 & 1 & 2 & 2\\ \hline
max codim $L$ & 4 & 3 & 3 & 2 & 2 & 2 & 2 & 2 & 2 & 2 & 2 & 2 & 3 & 2 & 2 & 2 & 2 & 2 & 2 & 2 & 2 & 2 & 2 & 2 & 2 \\ \hline
Length $L$ 		& 4 & 4 & 4 & 4 & 4 & 4 & 4 & 4 & 4 & 4 & 4 & 4 & 4 & 4 & 4 & 3 & 4 & 4 & 4 & 4 & 4 & 4 & 4 & 3 & 4\\ \hline
$\dim L$ 		& 4 & 3 & 3 & 3 & 3 & 3 & 3 & 3 & 3 & 3 & 3 & 3 & 3 & 3 & 3 & 3 & 3 & 3 & 2 & 3 & 3 & 3 & 3 & 3 & 3\\ \hline
Breadth $L$ 	& 4 & 3 & 3 & 3 & 3 & 3 & 3 & 3 & 3 & 3 & 3 & 3 & 3 & 3 & 3 & 3 & 2 & 3 & 2 & 3 & 3 & 3 & 3 & 3 & 3\\ \hline
Cohen-Macaulay 	& X & X & X &   &   &   &   &   & X &   &   &   & X &   &   &   & X &   & X &   &   &   &   &   & \\ \hline
Strongly generic 		& X & X & X &   & X & X &   &   & X & X &   &   &   &   &   &   & X &   & X &   &   &   &   &   &   \\ \hline

\end{tabular}
\vspace*{2ex}
\caption{First 25} \label{table:first25}

\hspace{2cm}

\begin{tabular}{*{26}{c}}\hline
No. 			& 26 & 27 & 28 & 29 & 30 & 31 & 32 & 33 & 34 & 35 & 36 & 37 & 38 & 39 & 40 & 41 & 42 & 43 & 44 & 45 & 46 & 47 & 48 & 49 & 50\\ \hline
$|L|$ 			& 10 & 10 & 10 & 10 & 10 & 10 & 10 & 10 & 10 & 10 & 9  & 9  & 9  & 9  & 9  & 9  & 9  & 9  & 8  & 8  & 8  & 8  & 7  & 7  & 6\\ \hline
pdim $S/I$ 		& 2  & 2  & 3  & 3  & 2  & 3  & 3  & 3  & 3  & 3  & 2  & 2  & 2  & 2  & 3  & 3  & 2  & 2  & 2  & 2  & 2  & 2  & 2  & 2  & 2\\
spdim $S/I$ 	& 2  & 2  & 3  & 3  & 2  & 3  & 3  & 3  & 3  & 3  & 2  & 2  & 2  & 2  & 3  & 3  & 2  & 2  & 2  & 2  & 2  & 2  & 2  & 2  & 2\\ \hline
pdim $I$ 		& 1  & 1  & 2  & 2  & 1  & 2  & 2  & 2  & 2  & 2  & 1  & 1  & 1  & 1  & 2  & 2  & 1  & 1  & 1  & 1  & 1  & 1  & 1  & 1  & 1\\
spdim $I$ 		& 1  & 1  & 1  & 1  & 1  & 1  & 1  & 1  & 1  & 2  & 1  & 1  & 1  & 1  & 1  & 1  & 1  & 1  & 1  & 1  & 1  & 1  & 1  & 1  & 1\\ \hline
max codim $L$ & 2  & 2  & 2  & 2  & 2  & 2  & 2  & 2  & 2  & 3  & 2  & 2  & 2  & 2  & 2  & 2  & 2  & 2  & 2  & 2  & 2  & 2  & 2  & 2  & 2 \\ \hline

Length $L$ 		& 4  & 4  & 4  & 4  & 4  & 3  & 4  & 4  & 3  & 3  & 4  & 4  & 4  & 3  & 3  & 3  & 4  & 3  & 4  & 3  & 3  & 3  & 3  & 3  & 2\\ \hline
$\dim L$ 		& 2  & 3  & 3  & 3  & 2  & 3  & 3  & 3  & 3  & 3  & 2  & 2  & 2  & 3  & 3  & 3  & 2  & 2  & 2  & 2  & 2  & 2  & 2  & 2  & 2\\ \hline
Breadth $L$ 	& 2  & 2  & 3  & 3  & 2  & 3  & 3  & 3  & 2  & 3  & 2  & 2  & 2  & 2  & 3  & 3  & 2  & 2  & 2  & 2  & 2  & 2  & 2  & 2  & 2\\ \hline
Cohen-Macaulay	& X  & X  &    &    & X  &    &    &    &    & X  & X  & X  & X  & X  &    &    & X  & X  & X  & X  & X  & X  & X  & X  & X \\ \hline
Strongly generic 		&    & X  &    &    &    &    &    &    &    &    &    &    &    & X  &    &    &    &    &    &    &    &    &    &    &   \\ \hline

\end{tabular}

\vspace*{2ex} \caption{Last 25} \label{table:last25}
\end{sidewaystable}

\FloatBarrier

\bibliographystyle{alpha}
\bibliography{LCM2}

\end{document}